\documentclass[11pt]{amsart}

\usepackage{amsfonts}
\usepackage{amssymb}
\usepackage{amsmath,mathtools}
\usepackage{latexsym}
\usepackage{amscd}
\usepackage{xypic}
\usepackage{mathrsfs}
\usepackage{enumitem} 
\usepackage{braket}
\usepackage[margin=1.3in]{geometry}
\usepackage{esint}
\usepackage{amsthm}
\usepackage{accents}
\usepackage[protrusion=true,expansion,stretch=5]{microtype}
\usepackage{tikz}
\usetikzlibrary{snakes}

\usepackage[pdftex]{hyperref}

\setlist{itemsep=3pt}

\allowdisplaybreaks

\newtheorem{prop}{Proposition}
\newtheorem{theo}[prop]{Theorem}
\newtheorem{lemm}[prop]{Lemma}

\newtheorem{conj}[prop]{Conjecture}

\theoremstyle{definition}

\newtheorem{rema}[prop]{Remark}

\newcommand{\RR}{\mathbb{R}}

\DeclareMathOperator{\area}{area}

\DeclareMathOperator{\Vol}{Vol}

\let\oldmarginpar\marginpar
\renewcommand\marginpar[1]{\-\oldmarginpar[\raggedleft\footnotesize #1]%
{\raggedright\footnotesize #1}}

\DeclareMathOperator{\Scal}{scal}

\usepackage{graphicx}

\allowdisplaybreaks

\author{Otis Chodosh}
\address
  {Department of Mathematics, Stanford University, Stanford, CA 94305, USA
}
\email{ochodosh@stanford.edu}

\author{Noa Vikman}
\address
  {Department of Mathematics\\ University of Fribourg\\ Chemin du Mus\'ee 23\\ 1700 Fribourg, Switzerland}
\email{noa.vikman@unifr.ch}

\date{\today}

\title{Macroscopic scalar curvature bounds on surfaces}

\begin{document}

\begin{abstract}
    We prove the generalized Geroch conjecture, and the corresponding hyperbolic conjecture, in dimension two.
\end{abstract}

\maketitle 

\section*{AI usage statement}
This note is a presentation of results obtained by ChatGPT-5.6 Pro without significant assistance from the authors. We have checked and reworked the proof but the fundamental strategy remains close to the original output. This article does not contain AI-written text.

\section{Introduction}

In his article on metaphors in systolic geometry \cite{Guth2010}, Guth describes the following macroscopic version of Geroch's conjecture that he attributes to Gromov \cite{Gromov1986}.

\begin{conj}[Generalized Geroch conjecture] \label{conj:torus}
    Fix $r>0$. The $n$-dimensional torus does not admit a Riemannian metric with $\Scal_r>0$.
\end{conj}

Here, $\Scal_r>0$ is equivalent to the property that all balls of radius $r$ in the universal cover have volumes strictly less than $\omega_nr^n$, where $\omega_n$ is the volume of a unit ball in $\mathbb R^n$. Note that as $r\to 0$, the macroscopic scalar curvature converges to the usual scalar curvature, and as such this conjecture is a generalization of the non-existence of a metric with positive scalar curvature on the $n$-dimensional torus. For more context, see \cite{SchoenYau,GL,Stern,BHHSZ,BrendleWang,BiZhu}.  

Guth has also stated a similar conjecture for the hyperbolic case \cite{Guth2011}.
\begin{conj}\label{conj:hyperbolic}
    Let $(M^n,g_0)$ be a closed hyperbolic manifold, and let $g$ be another metric on $M$ with $\Vol_g(M)<\Vol_{g_0}(M)$. Then the following inequality holds for all $r>0$
    \begin{equation*}
    \max_{\tilde x\in \tilde M}\Vol_{\tilde g}(B_{\tilde g}(\tilde x,r))> \Vol_{\mathbb H^n}(B(o,r)).
    \end{equation*}
\end{conj}
Here, $B_{\tilde g}(\tilde x,r)$ denotes a ball centered at $\tilde x$ of radius $r$, and $\tilde M$ is the universal cover of $M$, with lifted metric $\tilde g$. This is the macroscopic generalization of a conjecture of Schoen \cite[p.\ 127]{Schoen:mont}. Non-sharp versions of Conjectures~\ref{conj:torus} and \ref{conj:hyperbolic} have been obtained by several authors, cf.\ 
\cite{BuragoIvanov,Guth2010a,Karam2015,BalacheffKaram,BraunSauer2021,Sabourau2022,KS,Alpert2026,Zhang}.

 In this note, we give proofs of Conjecture~\ref{conj:torus} and Conjecture~\ref{conj:hyperbolic} when $n=2$. In particular, we have the following theorems.

\begin{theo} \label{theo:main-torus}
    Let $(\Sigma,g)$ be a closed oriented Riemannian surface of genus $k= 1$. Then, for every $r>0$,
    \begin{equation}\label{eq:main-thm-1}
    \max_{\tilde x\in \tilde \Sigma}\area_{\tilde g}(B_{\tilde g}(\tilde x,r))\geq \pi r^2,
    \end{equation}
    with equality for a given $r>0$ if and only if $g$ has constant curvature 0.
\end{theo}
\begin{theo} \label{theo:main-hyperbolic}
Let $(\Sigma,g)$ be a closed oriented Riemannian surface of genus $k\geq2$, normalized such that $\area_g(\Sigma) = \area_{g_0}(\Sigma)$ where $g_0$ is a metric of constant curvature $-1$. Then, for every $r>0$,
\begin{equation}\label{eq:main-thm-2}
    \max_{\tilde x\in \tilde \Sigma}\area_{\tilde g}(B_{\tilde g}(\tilde x,r))\geq 2\pi(\cosh{r}-1), 
\end{equation}
with equality for a given $r>0$ if and only if $g$ has constant curvature $-1$.
\end{theo}

The methods used here depend on the uniformization theorem in a fundamental way and thus do not seem to extend to higher dimensions (even to give a non-sharp estimate).

\subsection{Outline of proof strategies}
Both results follow similar strategies so we focus on Theorem \ref{theo:main-torus}. We prove \eqref{eq:main-thm-1} by showing that the average value of $\area_{\tilde g}(B_{\tilde g}(\tilde{x},r))$ over $\Sigma$ \emph{with respect to a suitable measure} (not the area measure) is at least $\pi r^2$. To do this, we let $\Omega(x)$ be the star-shaped region in $\RR^2$ containing all Euclidean segments based at a lifted point $\tilde x$ with $\tilde g$-length at most $r$ (this is clearly a subset of the $\tilde g$-ball of radius $r$). Then we write the area of $\Omega(x)$ as an integral over a fiber of the $g_0$-unit sphere bundle and show that the average of the resulting area element with respect to some measure $\nu$ is bounded below. 

This is done in two steps. First, we define $\Psi_s$ to be a $1$-parameter family of diffeomorphisms of the $g_0$-unit tangent bundle whose flowlines are (lifts of) $g_0$-geodesics reparametrized to have $g$-unit speed. There exists a natural measure $\nu$ on the $g_0$-unit tangent bundle that's invariant under $\Psi_s$ and if we average the area of $\Omega(x)$ with respect to (the projection of) $\nu$ then we can use invariance of $\nu$ in a change of variables argument. Second, the reversibility of the $\Psi_s$ flow gives a different way to write this averaged area element whose integrand turns out to be the ``reciprocal'' of the first expression. These two expressions can be combined using AM-GM to give the desired integrated lower bound on the area element, which yields---after applying Fubini---a lower bound for the average area of $\Omega(x)$. 

\begin{rema}
The use of the unit tangent bundle and the strategy of averaging the area of a region with respect to a background measure has appeared in the work of Katz--Sabourau \cite{KS}, although the precise choice of measure (and the reversal argument) used here are different. See also \cite{BCIK,KN}.
\end{rema}

\subsection{Acknowledgements} N.V. was supported by the Swiss National Science Foundation Grant 212867. O.C. was
partially supported by a Terman Fellowship and an NSF grant (DMS-2304432). He is grateful to DeepMind for access to its AlphaEvolve system which he used (before the proofs communicated here were discovered) to search for a counterexample to Theorem \ref{theo:main-torus}. We are grateful to Yevgeny Liokumovich for putting us in contact as well as for his interest in this note. 

\section{Proofs of Theorem~\ref{theo:main-torus} and Theorem~\ref{theo:main-hyperbolic}}

Let $(\Sigma,g)$ be a closed oriented Riemannian surface of genus $k\geq 1$. By uniformization there is a smooth conformal factor $u\colon \Sigma\to \mathbb R^+$ such that $g=u^2g_0$ where $g_0$ is a metric of constant curvature $0$ or $-1$. Let $d\mu_0$ denote the area measure for $g_0$. In the higher genus case we assume that $\area_g(\Sigma) = \area_{g_0}(\Sigma)$. 

\subsection{The star-shaped domains} 
Fix $r>0$ for the remainder of the paper. Given some point $z=(x,\theta)\in T^1\Sigma$ from the unit tangent bundle, let $\gamma_z(t)$ denote a unit speed $g_0$-geodesic in the $t$-parameter starting at $x$ in the direction of $\theta$. Let $T_s(z)$ denote the reparametrization such that $\gamma_z(T_s(z))$ is unit speed in the $s$-parameter with respect to the metric $g$. We define $\eta_z(s)=\gamma_z(T_s(z))$.

Now for some $x\in \Sigma$, take some lift $\tilde x\in \tilde \Sigma$ in the universal cover. We define the star-shaped domain
\[
\Omega(x)\coloneqq \{\tilde\eta_{(x,\theta)}(s): s\in [0,r], \ \theta\in S^1\}\subset \tilde\Sigma
\]
where the lifts of $\eta$ are such that $\tilde\eta_{(x,\theta)}(0)=\tilde x$. Since the $\tilde g$-distance between endpoints of the curves is at most their $\tilde g$-lengths, $\Omega(x)\subset B_{\tilde g}(\tilde x,r)$. Let us compute the area of the domain using polar coordinates in the universal cover:
\[
\area_{\tilde g}(\Omega(x)) = \int_{S^1}A(x,\theta)\,d\theta
\]
where,
\begin{equation}\label{eq:A-integral}
A(x,\theta) = \int_{0}^{T_r(x,\theta)}u(\gamma_{(x,\theta)}(t))^2t\, dt 
\quad \text{or} \quad \int_{0}^{T_r(x,\theta)}u(\gamma_{(x,\theta)}(t))^2 \sinh{t}\, dt
\end{equation}
for the torus case and the higher genus case respectively. We will show that
\[
\max_{ x\in \Sigma}\area_{\tilde g}(\Omega(x)) \geq \pi r^2 \quad \text{or} \quad 
\max_{ x\in \Sigma}\area_{\tilde g}(\Omega(x)) \geq 2\pi(\cosh r-1).
\]
Since $\Omega(x)\subset B_{\tilde g}(\tilde x,r)$ this implies (\ref{eq:main-thm-1}) and (\ref{eq:main-thm-2}).

\subsection{Geodesic flow and the invariant measure}
Let $X$ be the geodesic vector field on the $g_0$-unit tangent bundle $T^1\Sigma$, and let $U=u\circ \pi$ where $\pi\colon T^1\Sigma\to \Sigma$ is the projection map. Consider the scaled vector field $U^{-1}X$. Let $\phi_t$ and $\Psi_s$ describe the flow in $T^1\Sigma$ along $X$ and $U^{-1}X$ respectively. Note that $\pi(\phi_t(z)) = \gamma_z(t)$ and $\pi(\Psi_s(z)) = \eta_z(s)$. Let $L$ be the Liouville volume form on $T^1\Sigma$, whose induced measure $dL$ is invariant with respect to the geodesic flow of $\phi_t$. We now consider the scaled form $\nu =UL$.
\begin{lemm}
    The measure $d\nu$ is invariant with respect to the flow $\Psi_s$.
\end{lemm}
\begin{proof}
    We compute the Lie derivative using Cartan's formula and Liouville's theorem,
    \[\mathcal L_{U^{-1}X}(\nu)= \mathcal L_{U^{-1}X}(UL)=d\iota_{U^{-1}X}(UL) = d\iota_X(L) = \mathcal L_X( L) = 0\]
    implying that $d\nu$ is invariant under the flow $\Psi_s$.
\end{proof}
Note that since integrating with $dL$ can be done fiberwise,
\[
\int_{T^1\Sigma} A(z)\,d\nu(z) = \int_\Sigma\int_{S^1_x}A(x,\theta_x)u(x)\, d\theta_x  d\mu_0(x) = \int_\Sigma \area_{\tilde g}(\Omega(x))u(x)\,d\mu_0(x).
\]
For a description of the geodesic vector field, the Liouville measure and its invariance to the geodesic flow, see \cite[Chapter VII]{Chavel2006}.

\subsection{The proof of Theorem~\ref{theo:main-torus}}
To start, we change variables in (\ref{eq:A-integral}) from $t$ to $s$, where $t=T_s(z)$, using,
\[
\frac{dt}{ds} = \frac{1}{U(\Psi_s(z))}, \qquad 
t(s) = \int_0^s\frac{d\sigma}{U(\Psi_\sigma(z))}\]
so that
\[
A(z) = \int_0^r \int_0^s \frac{U(\Psi_s(z))}{U(\Psi_\sigma(z))} \,d\sigma ds. 
\]
We now rewrite the following integral, using the invariance of $d\nu$ to the flow $\Psi_s$,
\begin{align*}
    \int_{T^1\Sigma}A(z)\,d\nu(z) &= \int_0^r\int_0^s\int_{T^1\Sigma}\frac{U(\Psi_s(z))}{U(\Psi_\sigma(z))}\,d\nu(z)d\sigma ds\\
    &= \int_0^r\int_0^s\int_{T^1\Sigma}\frac{U(\Psi_{s-\sigma}(z))}{U(z)}\,d\nu(z)d\sigma ds.
\end{align*}
Next, in the key step of this proof, we reverse the direction of the flow. For $z=(x,\theta)$ let $\bar z= (x,-\theta)$ and note that $\Psi_s(\bar z) = \overline{\Psi_{-s}(z)}$. Then, using invariance of $\nu$ under the reversal map we have
\begin{align*}
    \int_{T^1\Sigma}A(\bar z)\,d\nu(z)&=\int_0^r\int_0^s\int_{T^1\Sigma}\frac{U(\Psi_{-s}(z))}{U(\Psi_{-\sigma}(z))}\,d\nu(z)d\sigma ds\\
    &=\int_0^r\int_0^s\int_{T^1\Sigma}\frac{U(z)}{U(\Psi_{s-\sigma}(z))}\,d\nu(z)d\sigma ds.
\end{align*}
Taking advantage of the two reciprocal integrands, with the fact that $a+a^{-1}\geq 2$ for all $a>0$, we can estimate
\begin{align*}
    \int_\Sigma \area_{\tilde g}(\Omega(x))u(x)\,d\mu_0(x) &= \int_{T^1\Sigma}A(z)\,d\nu(z) \\&= \frac{1}{2}\int_{T^1\Sigma} (A(z)+A(\bar z))d\nu(z)\\
    &= \frac{1}{2}\int_0^r\int_0^s\int_{T^1\Sigma}\frac{U(z)}{U(\Psi_{s-\sigma}(z))}+\frac{U(\Psi_{s-\sigma}(z))}{U(z)}\,d\nu(z)d\sigma ds\\
    &\geq \int_0^r\int_0^s\int_{T^1\Sigma}d\nu(z)\,d\sigma ds = \nu(T^1\Sigma)\frac{r^2}{2}.
\end{align*}
This proves the volume bound (\ref{eq:main-thm-1}) in Theorem \ref{theo:main-torus} once we note that
\[\nu(T^1\Sigma) = 2\pi\int_\Sigma u(x)\,d\mu_0(x).\]
Finally, if the inequality (\ref{eq:main-thm-1}) is an equality, it means that $U$ is invariant with respect to the flow $\Psi_s$ for all $s \in (0,r)$. Differentiating this at $s=0$ we obtain $XU(z) = 0$. As such, $u$ is constant along every $g_0$-geodesic. This means that $u$ is constant, so $g$ has constant curvature 0.

\subsection{The proof of Theorem~\ref{theo:main-hyperbolic}}

As in the previous proof, we begin with a change of variables from $t$ to $s$ of (\ref{eq:A-integral}) where $t=T_s(z)$. In this case we obtain
\begin{align*}
A(z) &= \int_0^r U(\Psi_s(z))\sinh( T_s(z)) \,ds\\
&= \int_0^r\int_0^s
\frac{U(\Psi_s(z))}{U(\Psi_\sigma(z))}
\cosh\big(T_{s-\sigma}(\Psi_\sigma(z))\big)
\,d\sigma ds.
\end{align*}
In the second step above, we use the following identities
\[
\sinh(T_s(z)) = \int_0^s\frac{\cosh(T_{s}(z)-T_{\sigma}(z))}{U(\Psi_\sigma(z))}\,d\sigma, \qquad T_s(z)-T_\sigma(z) = T_{s-\sigma}(\Psi_\sigma(z)).
\]
We proceed by again using the invariance property of $d\nu$ together with a change in direction of the flow. In this case,
\begin{align*}
    \int_{T^1\Sigma}\frac{U(\Psi_s(z))}{U(\Psi_\sigma(z))}\cosh(T_{s-\sigma}(\Psi_\sigma(z)))\,d\nu(z) &= 
    \int_{T^1\Sigma}\frac{U(\Psi_{s-\sigma}(z))}{U(z)}\cosh(T_{s-\sigma}(z))\,d\nu(z)\\
    &= \int_{T^1\Sigma}\frac{U(y)}{U(\Psi_{s-\sigma}(y))}\cosh(T_{s-\sigma}(y))\,d\nu(y)
\end{align*}
where the second equality uses $\Psi_{s-\sigma}(z) = \bar y$ and the fact that 
\[
T_{s-\sigma}(\Psi_{\sigma-s}(\bar y)) = T_{s-\sigma}(\overline{\Psi_{s-\sigma}(y)}) = T_{s-\sigma}(y).
\]
Just as in the previous proof, we can use these reciprocal integrands to find an estimate
\begin{equation}\label{eq:hyperbolic-estimate}
    \int_{T^1\Sigma}A(z)\,d\nu(z) \geq\int_0^r\int_0^s\int_{T^1\Sigma}\cosh(T_{s-\sigma}(z))\,d\nu(z)d\sigma ds.
\end{equation}
Using Jensen's inequality, Cauchy-Schwarz and the area normalization of $g$ (see a more detailed explanation below), we get,
\begin{equation}\label{eq:cosh-estimate}
\int_{T^1\Sigma}\cosh(T_{s-\sigma}(z))\,d\nu(z) \geq \nu(T^1\Sigma)\cosh(s-\sigma).
\end{equation}
By combining (\ref{eq:cosh-estimate}) with
 (\ref{eq:hyperbolic-estimate}) we have
\[
\int_{\Sigma}\area_{\tilde g}(\Omega(x))u(x)\,d\mu_0(x)\geq \nu(T^1\Sigma)(\cosh r -1).
\]
This completes the proof of the volume bound (\ref{eq:main-thm-2}) just as in the case of the torus. The rigidity of the metric when the volume bound is an equality also follows for the same reason as in the torus case.

To see why (\ref{eq:cosh-estimate}) holds, begin with Jensen's inequality
\begin{align*}
\int_{T^1\Sigma}\cosh(T_{s-\sigma}(z))\,d\nu(z)&\geq \nu(T^1\Sigma)\cosh\left(\nu(T^1\Sigma)^{-1}\int_{T^1\Sigma}T_{s-\sigma}(z)\,d\nu(z)\right).
\end{align*}
Then note that, by the invariance of $d\nu$ to $\Psi_s$, and the area-normalization of $(\Sigma,g)$,
\begin{align*}
    \int_{T^1\Sigma}T_{s-\sigma}(z)\,d\nu(z) &= \int_{T^1\Sigma}\int_0^{s-\sigma}\frac{1}{U(\Psi_t(z))}\,dtd\nu(z)\\
    &=\int_0^{s-\sigma}\int_{T^1\Sigma}\frac{1}{U(z)}\,d\nu(z)dt\\
    &= 2\pi(s-\sigma)\area_{g_0}(\Sigma)\\
    &= 2\pi(s-\sigma)\area_g(\Sigma).
\end{align*}
Finally, by Cauchy-Schwarz inequality,
\[\nu(T^1\Sigma)=2\pi\int_{\Sigma}u\,d\mu_0\leq 2\pi(\area_g(\Sigma)^{1/2}\area_{g_0}(\Sigma)^{1/2}) = 2\pi\area_g(\Sigma).\]
Since $[0,\infty)\ni x \mapsto \cosh x$ is increasing, the inequality  (\ref{eq:cosh-estimate}) follows by combining all the above.

\end{document}